\documentclass[review]{elsarticle}

\usepackage{lineno,hyperref}
\modulolinenumbers[5]
\usepackage{epsfig}
\usepackage{geometry}
\usepackage{color}
\usepackage{amsmath}
\usepackage{amssymb}
\usepackage{amsthm}
\usepackage{amsfonts}
\usepackage{enumitem}
\usepackage{dsfont}

\newtheorem{lem}{Lemma}[section]
\newtheorem{thm}{Theorem}[section]

\newtheorem{remark}{Remark}[section]
\newtheorem{corr}{Corollary}[section]

\journal{arXiv}

\begin{document}

\begin{frontmatter}

\title{A new type of the Gronwall-Bellman inequality and its application to fractional stochastic differential equations}



\author{Qiong Wu\corref{mycorrespondingauthor}}
\cortext[mycorrespondingauthor]{Corresponding author}
\ead{Qiong.Wu@tufts.edu}

\address{Tufts University, \\Department of Mathematics, 503 Boston Avenue, Medford, MA 02155, USA.}

\begin{abstract}
This paper presents a new type of Gronwall-Bellman inequality, which arises from a class of integral equations with a mixture of nonsingular and singular integrals. The new idea is to use a binomial function to combine the known Gronwall-Bellman inequalities for integral equations having nonsingular integrals with those having singular integrals. Based on this new type of Gronwall-Bellman inequality, we investigate the existence and uniqueness of the solution to a fractional stochastic differential equation (SDE) with fractional order~$0 < \alpha < 1$. This result generalizes the existence and uniqueness theorem related to fractional order $\frac{1}{2} < \alpha < 1$ appearing in \cite{pedjeu2012stochastic}. Finally, the fractional type Fokker-Planck-Kolmogorov equation associated to the solution of the fractional SDE is derived using It\^o's formula.
\end{abstract}

\end{frontmatter}

\linenumbers
\section{Introduction}
It is well known that integral inequalities are instrumental in studying the qualitative analysis of solutions to differential and integral equations~\cite{ames1997inequalities}. Among these inequalities, the distinguished Gronwall-Bellman type inequality from~\cite{bellman1963differential}, and its associated extensions,~\cite{pachpatte1975some, lipovan2000retarded, agarwal2005generalization, mao1989lebesgue}, are capable of affording explicit bounds on solutions of a class of linear differential equations with integer order. The following lemma concerns a standard Gronwall-Bellman inequality in~\cite{corduneanu2008principles} for a differential equation with order one or equivalently an integral equation with nonsingular integrals.
\begin{lem}\label{stadgronwall}
	Suppose $h(t)$, $k(t)$ and $x(t)$ are continuous functions on $t_{0} \leq t < T, 0 < T\leq \infty,$ with $k(t) \geq 0$. If $x(t)$ satisfies
	\begin{eqnarray*}
		x(t) \leq h(t) + \int_{t_{0}}^{t}k(s)x(s)\mathrm{d}s,
	\end{eqnarray*} 
	then
	\begin{eqnarray*}
	x(t) \leq h(t) + \int_{t_{0}}^{t}h(s)k(s)\exp\Bigg [\int_{s}^{t}k(u)\mathrm{d}u\Bigg ]\mathrm{d}s.
	\end{eqnarray*}
Moreover, if $h(t)$ is nondecreasing, then
	\begin{eqnarray*}
		x(t) \leq h(t)\exp\Bigg [\int_{t_{0}}^{t}k(s)\mathrm{d}s\Bigg ].
	\end{eqnarray*}
\end{lem}

In order to investigate the qualitative properties of solutions to  differential equations of fractional order, there are several generalizations of Gronwall-Bellman inequalities developed by many researchers~\cite{ye2007generalized, zheng2013some, lazarevic2009finite, atici2012gronwall}. Let's recall the following generalized Gronwall-Bellman inequality proposed in~\cite{ye2007generalized} for a fractional differential equation with order $\beta > 0$ or equivalently an integral equation with singular integrals.
\begin{lem}\label{fracgronwall}
	Suppose $\beta > 0$, $a(t)$ is a nonnegative function which is locally integrable on $0 \leq t < T, 0 < T\leq \infty$, and $g(t)$ is a nonnegative, nondecreasing continuous function defined on $0 \leq t < T$ with $g(t) \leq M~(constant)$. If $u(t)$ is nonnegative and locally integrable on $0 \leq t < T $ with
	\begin{eqnarray*}
		u(t) \leq a(t) + g(t)\int_{0}^{t}(t -s)^{\beta - 1}u(s)\mathrm{d}s
	\end{eqnarray*} 
	on this interval, then
	\begin{eqnarray*}
	u(t) \leq a(t) + \int_{0}^{t}\Bigg [ \sum_{n = 1}^{\infty} \frac{(g(t)\Gamma(\beta))^{n}}{\Gamma(n\beta)}(t - s)^{\{n\beta - 1\}}a(s)\Bigg]\mathrm{d}s,
	\end{eqnarray*}
where $\Gamma(t)$ is the gamma function. Furthermore, if $a(t)$ is nondecreasing on $0 \leq t < T$, then
	\begin{eqnarray*}
		u(t) \leq a(t)E_{\beta}(g(t)\Gamma(\beta)t^{\beta}),
	\end{eqnarray*} 
where $E_{\beta}(z)$ is the Mittag-Leffler function defined by $E_{\beta}(z) = \sum_{k=0}^{\infty}\frac{z^{k}}{\Gamma(k\beta +1)}$ for $z > 0$.
\end{lem}

From many real applications, such as in physics, theoretical biology, and mathematical finance, there is substantial interest in a class of fractional SDEs~\cite{mandelbrot1968fractional, jumarie2005solution, pedjeu2012stochastic}. The fractional SDEs take the form
\begin{eqnarray}\label{fraceq1}
	\mathrm{d}x(t) = b(t,x(t))\mathrm{d}t + \sigma_{1}(t, x(t))\mathrm{d}t^{\alpha} + \sigma_{2}(t, x(t))\mathrm{d}B_{t},
\end{eqnarray}
where the initial value is $x(0) = x_{0}$, $0 < \alpha < 1$, and $B_{t}$ is the standard Brownian motion. According to~\cite{jumarie2005solution, pedjeu2012stochastic}, the integral equation corresponding to Eq.~\eqref{fraceq1} is
\begin{eqnarray}\label{fracinteg}
x(t) = x_{0} + \int_{0}^{t}b(s, x(s))\mathrm{d}s + \alpha\int_{0}^{t}(t - s)^{\alpha - 1}\sigma_{1}(s, x(s))\mathrm{d}s + \int_{0}^{t}\sigma_{2}(s, x(s))\mathrm{d}B_{s}. 
\end{eqnarray}
Since $0 < \alpha < 1$, there are nonsingular and singular  integrals in the integral equation Eq.~\eqref{fracinteg}. However, the above mentioned types of Gronwall-Bellman inequalities, such as Lemmas~\ref{stadgronwall} and~\ref{fracgronwall}, are not applicable to studying the qualitative properties of the solution to Eq.~\eqref{fraceq1} or~Eq.\eqref{fracinteg}.  

The first goal of this paper, presented in Section~\ref{mainineq}, is to derive a new type of Gronwall-Bellman inequality which is applicable to study the qualitative behaviors of the solution to the fractional SDE Eq.~\eqref{fraceq1} or the stochastic integral equation Eq.~\eqref{fracinteg}. The second goal, accomplished in Section~\ref{application}, is to apply the results from Section~\ref{mainineq} to investigate existence and uniqueness of the solution to the fractional SDE Eq.~\eqref{fraceq1} of order $0 < \alpha < 1$. Finally, in Section~\ref{fpe}, a fractional type Fokker-Planck-Kolmogorov equation associated to the solution of the fractional SDE Eq.~\eqref{fraceq1} is derived.
     
\section{Generalization of the Gronwall-Bellman Inequality}\label{mainineq}
In this section, we develop a new integral inequality, Eq.~\eqref{fracgrowineq} below, by verifying three claims. The first claim is established by using the method of induction and taking advantage of the binomial function; the second claim is verified by taking advantage of properties of the Gamma function; the third claim is verified by employing Gamma functions, Mittag-Leffler functions and exponential functions.  The established integral inequality is applicable to the fractional SDE Eq.\eqref{fraceq1} or the stochastic integral equation Eq.~\eqref{fracinteg}. Also this new integral inequality can be considered as a generalization of the integral inequalities in Lemmas~\ref{stadgronwall} and~\ref{fracgronwall}. 
\begin{thm}\label{maingrownineq}
Let $0 < \alpha < 1$ and consider the time interval $I = [0, T)$, where $T\leq~\infty$. Suppose $a(t)$ is a nonnegative function, which is locally integrable on $I$ and $b(t)$ and $g(t)$ are nonnegative, nondecreasing continuous function defined on $I$, with both bounded by a positive constant,~$M$. If $u(t)$ is nonnegative, and locally integrable on $I$ and satisfies
	\begin{eqnarray}\label{condforineq}
		u(t) \leq a(t) + b(t)\int_{0}^{t}u(s)\mathrm{d}s + g(t)\int_{0}^{t}(t -s)^{\alpha - 1}u(s)\mathrm{d}s,
	\end{eqnarray} 
then
	\begin{eqnarray}\label{fracgrowineq}
		u(t) \leq a(t) + \sum_{n = 1}^{\infty}\sum_{i=0}^{n} \dbinom{n}{i}b^{n - i}(t)g^{i}(t)\frac{[\Gamma(\alpha)]^{i}}{\Gamma(i\alpha + n - i)}\int_{0}^{t}(t - s)^{\{i\alpha - (i+1-n)\}}a(s)\mathrm{d}s.
	\end{eqnarray}
\end{thm}
\begin{proof}
	Let $\phi$ be a locally integrable function and define an operator $B$ on $\phi$ as follows
	\begin{eqnarray*}
		B\phi(t) := b(t)\int_{0}^{t}\phi(s)\mathrm{d}s + g(t)\int_{0}^{t}(t - s)^{\alpha - 1}\phi(s)\mathrm{d}s, ~~~t\geq 0.
	\end{eqnarray*} 
 From the inequality, Eq.~\eqref{condforineq},
	\begin{eqnarray*}
		u(t) \leq a(t) + Bu(t).
	\end{eqnarray*} 
	This implies
	\begin{eqnarray}\label{iterineq}
		u(t) \leq \sum_{k = 0}^{n - 1}B^{k}a(t) + B^{n}u(t).
	\end{eqnarray} 
	In order to get the desired inequality, Eq.~\eqref{fracgrowineq}, from Eq.~\eqref{iterineq}, there are three claims to be verified. 
	
	The first claim provides a general bound on $B^{n}(t)$: 
	\begin{eqnarray}\label{guessiter}
		B^{n}u(t) \leq \sum_{i = 0}^{n}\dbinom{n}{i}b^{n - i}(t)g^{i}(t)\frac{[\Gamma(\alpha)]^{i}}{\Gamma(i\alpha + n - i)}\int_{0}^{t}(t - s)^{\{i\alpha - (i+1-n)\}}u(s)\mathrm{d}s.
	\end{eqnarray}
The method of induction will be used to verify the inequality in Eq.~\eqref{guessiter}. First let $n = 1$. Then the inequality in Eq.~\eqref{guessiter} is true. Now, suppose that the inequality, Eq.~\eqref{guessiter}, holds for $n = k$, and then compute $B^{n}$ when $n = k+1$,
	\begin{eqnarray*}
		\begin{aligned}
			B^{k+1}u(t) &= B(B^{k}u(t))\leq b(t)\int_{0}^{t}\sum_{i = 0}^{k}\dbinom{k}{i}b^{k-i}(s)g^{i}(s)\frac{[\Gamma(\alpha)]^{i}}{\Gamma(i\alpha + k - i)}\int_{0}^{s}(s - \tau)^{\{i\alpha - (i+1-k)\}}u(\tau)\mathrm{d}\tau\mathrm{d}s\\
			&+ g(t)\int_{0}^{t}(t - s)^{\alpha - 1}\sum_{i = 0}^{k}\dbinom{k}{i}b^{k-i}(s)g^{i}(s)\frac{[\Gamma(\alpha)]^{i}}{\Gamma(i\alpha+k-i)}\int_{0}^{s}(s - \tau)^{\{i\alpha - (i+1-k)\}}u(\tau)\mathrm{d}\tau\mathrm{d}s.\\
		\end{aligned}
	\end{eqnarray*}
Let
\begin{eqnarray*}
	C(t) := b(t)\int_{0}^{t}\sum_{i = 0}^{k}\dbinom{k}{i}b^{k-i}(s)g^{i}(s)\frac{[\Gamma(\alpha)]^{i}}{\Gamma(i\alpha + k - i)}\int_{0}^{s}(s - \tau)^{\{i\alpha - (i+1-k)\}}u(\tau)\mathrm{d}\tau\mathrm{d}s,
\end{eqnarray*} 
and
\begin{eqnarray*}
    G(t) :=  g(t)\int_{0}^{t}(t - s)^{\alpha - 1}\sum_{i = 0}^{k}\dbinom{k}{i}b^{k-i}(s)g^{i}(s)\frac{[\Gamma(\alpha)]^{i}}{\Gamma(i\alpha+k-i)}\int_{0}^{s}(s - \tau)^{\{i\alpha - (i+1-k)\}}u(\tau)\mathrm{d}\tau\mathrm{d}s.
\end{eqnarray*}
Then, compute $C(t)$ and $G(t)$ term by term to reach the desired inequality Eq.~\eqref{guessiter}. Since $b(t)$ and $g(t)$ are nonnegative and nondecreasing functions, 
	\begin{eqnarray}\label{computationofC(t)}
		\begin{aligned}
			C(t) &\leq \sum_{i = 0}^{k}b^{k-i+1}(t)g^{i}(t)\dbinom{k}{i}\frac{[\Gamma(\alpha)]^{i}}{\Gamma(i\alpha+k-i)}\int_{0}^{t}\int_{0}^{s}(s - \tau)^{\{i\alpha -(i+1-k)\}}u(\tau)\mathrm{d}\tau\mathrm{d}s\\
			&= \sum_{i = 0}^{k}b^{k-i+1}(t)g^{i}(t)\dbinom{k}{i}\frac{[\Gamma(\alpha)]^{i}}{\Gamma(i\alpha+k-i)}\int_{0}^{t}\int_{\tau}^{t}(s - \tau)^{\{i\alpha -(i+1-k)\}}u(\tau)\mathrm{d}s\mathrm{d}\tau\\
			&= \sum_{i = 0}^{k}b^{k-i+1}(t)g^{i}(t)\dbinom{k}{i}\frac{[\Gamma(\alpha)]^{i}}{\Gamma(i\alpha+k-i+1)}\int_{0}^{t}(t - \tau)^{\{i\alpha -(i-k)\}}u(\tau)\mathrm{d}\tau\\
			&= b^{k+1}(t)\dbinom{k}{0}\frac{1}{\Gamma(k+1)}\int_{0}^{t}(t - \tau)^{k}u(\tau)\mathrm{d}\tau\\
			&+ b(t)\sum_{i = 1}^{k}b^{k-i}(t)g^{i}(t)\dbinom{k}{i}\frac{[\Gamma(\alpha)]^{i}}{\Gamma(i\alpha+k-i+1)}\int_{0}^{t}(t - \tau)^{\{i\alpha -(i-k)\}}u(\tau)\mathrm{d}\tau.\\
		\end{aligned}
	\end{eqnarray}
	Similarly, compute $G(t)$ 
	\begin{eqnarray}\label{computationofG(t)}
		\begin{aligned}
			G(t) &\leq \sum_{i = 0}^{k}b^{k-i}(t)g^{i+1}(t)\dbinom{k}{i}\frac{[\Gamma(\alpha)]^{i}}{\Gamma(i\alpha+k-i)}\int_{0}^{t}(t - s)^{\alpha - 1}\int_{0}^{s}(s - \tau)^{\{i\alpha -(i+1-k)\}}u(\tau)\mathrm{d}\tau\mathrm{d}s\\
			&= \sum_{i = 0}^{k}b^{k-i}(t)g^{i+1}(t)\dbinom{k}{i}\frac{[\Gamma(\alpha)]^{i}}{\Gamma(i\alpha+k-i)}\int_{0}^{t}\int_{\tau}^{t}(t - s)^{\alpha - 1}(s - \tau)^{\{i\alpha -(i+1-k)\}}u(\tau)\mathrm{d}s\mathrm{d}\tau\\
			&= \sum_{i = 0}^{k}b^{k-i}(t)g^{i+1}(t)\dbinom{k}{i}\frac{[\Gamma(\alpha)]^{i+1}}{\Gamma((i+1)\alpha+k-i)}\int_{0}^{t}(t - \tau)^{\{(i+1)\alpha -(i+1-k)\}}u(\tau)\mathrm{d}\tau\\
			&= g^{k+1}(t)\dbinom{k}{k}\frac{[\Gamma(\alpha)]^{k+1}}{\Gamma((k+1)\alpha)}\int_{0}^{t}(t - \tau)^{\{(k+1)\alpha - 1\}}\mathrm{d}\tau\\ 
			&+ b(t)\sum_{i = 1}^{k}b^{k-i}(t)g^{i}(t)\dbinom{k}{i-1}\frac{[\Gamma(\alpha)]^{i}}{\Gamma(i\alpha+k-i+1)}\int_{0}^{t}(t - \tau)^{\{i\alpha -(i-k)\}}u(\tau)\mathrm{d}\tau.\\
		\end{aligned}
	\end{eqnarray}
	Combining Eq.~\eqref{computationofC(t)} and Eq.~\eqref{computationofG(t)} yields
	\begin{eqnarray*}
		\begin{aligned}
			B^{k+1}u(t) &= C(t) + G(t)\leq b^{k+1}(t)\dbinom{k}{0}\frac{1}{\Gamma(k+1)}\int_{0}^{t}(t - \tau)^{k}u(\tau)\mathrm{d}\tau\\ 
			&+ b(t)\sum_{i = 1}^{k}[\dbinom{k}{i-1} + \dbinom{k}{i}]b^{k-i}(t)g^{i}(t)\frac{[\Gamma(\alpha)]^{i}}{\Gamma(i\alpha+k-i+1)}\int_{0}^{t}(t - \tau)^{i\alpha - (i - k)}u(\tau)\mathrm{d}\tau\\
			&+ g^{k+1}(t)\dbinom{k}{k}\frac{[\Gamma(\alpha)]^{k+1}}{\Gamma((k+1)\alpha)}\int_{0}^{t}(t - \tau)^{(k+1)\alpha - 1}\mathrm{d}\tau\\
			&= \sum_{i=0}^{k+1}\dbinom{k+1}{i}b^{k+1-i}(t)g^{i}\frac{[\Gamma(\alpha)]^{i}}{\Gamma(i\alpha+k-i+1)}\int_{0}^{t}(t - \tau)^{i\alpha-(i-k)}u(\tau)\mathrm{d}\tau.
		\end{aligned}
	\end{eqnarray*}
	This implies that for any $n\in\mathbb{N}^{+}$, the first claim, Eq.~\eqref{guessiter}, holds. 

	The second claim shows that $B^{n}u(t)$ vanishes as $n$ increases. For each $t$ in $[0, T)$,
	\begin{eqnarray}\label{guess_second}
       B^{n}u(t) \to 0, ~as~n\to\infty.	
	\end{eqnarray}
	For the purpose of notation simplification during the proof of the second claim, define
	\begin{eqnarray*}
		H_{n}(t) := \sum_{i = 0}^{n}\dbinom{n}{i}b^{n - i}(t)g^{i}(t)\frac{[\Gamma(\alpha)]^{i}}{\Gamma(i\alpha + n - i)}\int_{0}^{t}(t - s)^{\{i\alpha - (i+1-n)\}}u(s)\mathrm{d}s.
	\end{eqnarray*} 
	Note that $\Gamma(x)$ is positive and decreasing on $(0, 1]$ but positive and increasing on $[2, \infty)$. Let $x_{i} = i\alpha + n - i$. Then, the sequence $x_{i}$ is decreasing over $[0, n]$ since $x_{i+1} - x_{i} = \alpha - 1 < 0$ when $i$ is an integer and $i\in [0, n]$. This means $x^{\min}_{i} = n\alpha$ and $x^{\max}_{i} = n$. Furthermore, for a fixed $\alpha$, there exists a large enough $n_{0}$ such that for any $n > n_{0}$, there is $n\geq\frac{2}{\alpha}$. So the sequence satisfies $x_{i} \geq 2$ for any integer $i\in[0, n]$ if $n$ is large enough. Thus, for any $i\in[0, n]$, $\Gamma(x^{\min}_{i}) < \Gamma(x_{i})$  and 
	\begin{eqnarray*}
		H_{n}(t) \leq \frac{1}{\Gamma(n\alpha)}\sum_{i = 0}^{n}\dbinom{n}{i}b^{n - i}(t)g^{i}(t)[\Gamma(\alpha)]^{i}\int_{0}^{t}(t - s)^{\{i\alpha - (i+1-n)\}}u(s)\mathrm{d}s, ~n>n_{0}.
	\end{eqnarray*}
	Also for $\alpha \in (0, 1)$, $\Gamma(\alpha) > 1$. Therefore, 
	\begin{eqnarray*}
		H_{n}(t) \leq \frac{[\Gamma(\alpha)]^{n}}{\Gamma(n\alpha)}\sum_{i = 0}^{n}\dbinom{n}{i}b^{n - i}(t)g^{i}(t)\int_{0}^{t}(t - s)^{\{i\alpha - (i+1-n)\}}u(s)\mathrm{d}s.
	\end{eqnarray*}	 
	Let $y_{i} = i\alpha + n - i - 1$. Similar to the sequence $x_{i}$, there is $y^{\min}_{i} = n\alpha - 1 \geq 1$ for a large enough $n$ and $y^{\max}_{i} = n - 1$. Since $t\in[0, T)$, split the interval $[0, T)$ into two subintervals $[0, 1]$ and $[1, T)$. For $t\in[0, 1]$, $(t - s)^{y_{i}} \leq t^{y_{min}} = t^{n\alpha - 1}$ while if $t\in[1, T)$, $(t - s)^{y_{i}}\leq t^{y^{\max}_{i}} = t^{n}$. Thus,  
	
	\begin{eqnarray}\label{boundofoperator}
		\begin{aligned}
			H_{n}(t) &\leq \frac{[\Gamma(\alpha)]^{n}\max\{t^{n\alpha - 1}, t^{n}\}}{\Gamma(n\alpha)}\sum_{i = 0}^{n}\dbinom{n}{i}b^{n - i}(t)g^{i}(t)\int_{0}^{t}u(s)\mathrm{d}s\\		
			&= \frac{[\Gamma(\alpha)]^{n}\max\{t^{n\alpha - 1}, t^{n}\}}{\Gamma(n\alpha)}(b(t) + g(t))^{n}\int_{0}^{t}u(s)\mathrm{d}s.\\ 
		\end{aligned}
	\end{eqnarray} 
Notice that $b(t)$ and $g(t)$ are both bounded by a positive constant~$M$, i.e., $b(t)\leq M$ and $g(t) \leq M$, and $u(s)$ is locally integrable over $0\leq t < T$. This means that from Eq.~\eqref{boundofoperator}, $H_{n}(t) \to 0$ as $n\to\infty$ because the Gamma function, $\Gamma(n\alpha)$, is growing faster than a power function. Therefore, the second claim, Eq.~\eqref{guess_second}, is verified since $B^{n}u(t) \leq H_{n}(t)$ for any $n\in\mathbb{N}^{+}$.
	
The third claim establishes that the right hand side (RHS) of Eq.~\eqref{fracgrowineq} exists on $0 \leq t <T$. In order to show this statement, we first prove that for $0 \leq t < T$, the following infinite sum of sequences denoted by $L(t; \tau)$ is convergent.
	\begin{eqnarray}\label{mixfunc}
		\begin{aligned}
			L(t; \tau) &:= \sum_{n=0}^{\infty}\sum_{i=0}^{n}\dbinom{n}{i}b^{n-i}(t)g^{i}(t)\frac{[\Gamma(\alpha)]^{i}}{\Gamma(i\alpha+n-i+1)}\tau^{i\alpha+n-i}\\
			&= \sum_{i=0}^{\infty}g^{i}(t)[\Gamma(\alpha)]^{i}\tau^{i\alpha}\sum_{n=i}^{\infty}\dbinom{n}{i}b^{n-i}(t)\frac{1}{\Gamma(i\alpha+n-i+1)}\tau^{n-i}\\
			&= \sum_{i=0}^{\infty}g^{i}(t)[\Gamma(\alpha)]^{i}\tau^{i\alpha}\frac{1}{\Gamma(i\alpha+1)}\sum_{n=i}^{\infty}\dbinom{n}{i}b^{n-i}(t)\frac{1}{(i\alpha+n-i)\cdots(i\alpha+1)}\tau^{n-i},\\
		\end{aligned}
	\end{eqnarray}
	where $(i\alpha+n-i)\cdots(i\alpha+1)$ is a product and it takes one if $(i\alpha+n-i) < i\alpha+1$. Let $k=n-i$, then compute
	\begin{eqnarray}\label{computationofproduct}
		\begin{aligned}
			\dbinom{n}{i}\frac{1}{(i\alpha + n - i)\cdots(i\alpha + 1)} &= \frac{(k+i)!}{i!k!}\frac{1}{(i\alpha + k)\cdots(i\alpha + 1)}
			= \frac{1}{k!}\frac{(k+i)\cdots(i+1)}{(i\alpha + k)\cdots(i\alpha + 1)}\\
			&=\frac{1}{\alpha^{k}}\frac{1}{k!}\frac{(k+1)\cdots(i+1)}{(k/\alpha + i)\cdots(i + 1/\alpha)}\leq \frac{1}{\alpha^{k}}\frac{1}{k!}.
		\end{aligned}
	\end{eqnarray}
	Substituting $k = n - i$ and Eq.~\eqref{computationofproduct} into Eq.~\eqref{mixfunc} gives
	\begin{eqnarray*}
		L(t; \tau) \leq \sum_{i=0}^{\infty}\frac{g^{i}(t)[\Gamma(\alpha)]^{i}\tau^{i\alpha}}{\Gamma(i\alpha+1)}\sum_{k=0}^{\infty}\frac{1}{\alpha^{k}}\frac{b^{k}(t)\tau^{k}}{k!} = E_{\alpha}(g(t)\Gamma(\alpha)\tau^{\alpha})\exp(\frac{1}{\alpha}b(t)\tau),
	\end{eqnarray*}
	which is finite for $0\leq t < T$. Furthermore, since $b(t) \leq M$ and $g(t)\leq M$, define 
	\begin{eqnarray*}
		L(M;\tau) :=  \sum_{i=0}^{\infty}\frac{M^{i}[\Gamma(\alpha)]^{i}\tau^{i\alpha}}{\Gamma(i\alpha+1)}\sum_{k=0}^{\infty}\frac{1}{\alpha^{k}}\frac{M^{k}\tau^{k}}{k!} = E_{\alpha}(M\Gamma(\alpha)\tau^{\alpha})\exp(\frac{1}{\alpha}M\tau),
	\end{eqnarray*}
	which means $L(M; \tau)$ is finite and $L(t; \tau)\leq L(M; \tau)$.
	Then, compute the RHS of Eq.~\eqref{fracgrowineq}
	\begin{eqnarray*}
		\begin{aligned}
		  RHS &= a(t) + \sum_{n = 1}^{\infty}\sum_{i=0}^{n} \dbinom{n}{i}b^{n - i}(t)g^{i}(t)\frac{[\Gamma(\alpha)]^{i}}{\Gamma(i\alpha + n - i + 1)}\int_{0}^{t}\frac{\mathrm{d}}{\mathrm{d}t}(t - s)^{\{i\alpha + n - i\}}a(s)\mathrm{d}s \\
		  & \leq a(t) + \sum_{n = 1}^{\infty}\sum_{i=0}^{n} \dbinom{n}{i}M^{n - i}M^{i}\frac{[\Gamma(\alpha)]^{i}}{\Gamma(i\alpha + n - i + 1)}\int_{0}^{t}\frac{\mathrm{d}}{\mathrm{d}t}(t - s)^{\{i\alpha + n - i\}}a(s)\mathrm{d}s \\
		  &= a(t) + \int_{0}^{t}\frac{\mathrm{d}L(M; t-s)}{\mathrm{d}t}a(s)\mathrm{d}s.
		\end{aligned}
	\end{eqnarray*}
Since the Mittag-Leffler function $E_{\alpha}(t^{\alpha})$ is an entire function in $t^{\alpha}$, see~\cite{gorenflo2002computation}, the exponential function~$\exp(t)$ is uniformly continuous in $t$, and both~$t^{\alpha - 1}$ and $a(t)$ are locally integrable over $0 \leq t < T$, the integral~$\int_{0}^{t}\frac{\mathrm{d}L(M; t-s)}{\mathrm{d}t}a(s)\mathrm{d}s$ is finite. This implies that the $RHS$ of Eq.~\eqref{fracgrowineq} is finite. So the last claim is also verified, thereby completing the proof.
\end{proof}

\begin{corr}\label{consisgronwall}
	Suppose the conditions in Theorem~\ref{maingrownineq} are satisfied and furthermore, $a(t)$ is nondecreasing on $0 \leq t < T$. Then 
	\begin{eqnarray*}
		u(t) \leq a(t)E_{\alpha}(g(t)\Gamma(\alpha)t^{\alpha})\exp(\frac{1}{\alpha}b(t)t).
	\end{eqnarray*}
\end{corr}
\begin{proof}
	From the proof of Theorem~\ref{maingrownineq}, 
	\begin{eqnarray*}
		\begin{aligned}
			u(t) &\leq a(t) + \sum_{n = 1}^{\infty}\sum_{i=0}^{n} \dbinom{n}{i}b^{n - i}(t)g^{i}(t)\frac{[\Gamma(\alpha)]^{i}}{\Gamma(i\alpha + n - i)}\int_{0}^{t}(t - s)^{\{i\alpha - (i+1-n)\}}a(s)\mathrm{d}s.			
		\end{aligned}
	\end{eqnarray*}
Since $a(t)$ is nondecreasing,
\begin{eqnarray*}
	\begin{aligned}
		u(t) &\leq a(t)\sum_{n = 0}^{\infty}\sum_{i=0}^{n} \dbinom{n}{i}b^{n - i}(t)g^{i}(t)\frac{[\Gamma(\alpha)]^{i}}{\Gamma(i\alpha + n - i)}\int_{0}^{t}(t - s)^{\{i\alpha - (i+1-n)\}}\mathrm{d}s\\
				&\leq a(t)\sum_{n = 0}^{\infty}\sum_{i=0}^{n} \dbinom{n}{i}b^{n - i}(t)g^{i}(t)\frac{[\Gamma(\alpha)]^{i}}{\Gamma(i\alpha + n - i+1)}t^{i\alpha + n - i}\\
				&\leq a(t)E_{\alpha}(g(t)\Gamma(\alpha)t^{\alpha})\exp(\frac{1}{\alpha}b(t)t).			
	\end{aligned}
\end{eqnarray*}
This completes the proof.
\end{proof}

\begin{remark}
	From Theorem~\ref{maingrownineq} and Corollary~\ref{consisgronwall}, we see that if $\alpha = 1$, Theorem~\ref{maingrownineq} and Corollary~\ref{consisgronwall} are the same as Lemma~\ref{stadgronwall}; while if $b(t)\equiv 0$, Theorem~\ref{maingrownineq} and Corollary~\ref{consisgronwall} become Lemma~\ref{fracgronwall}.
\end{remark}

\section{Existence and uniqueness of the solution to fractional SDEs}\label{application} 
In this section, using the main results from Section~\ref{mainineq}, we investigate the existence and uniqueness of the solution to the fractional SDE Eq.~\eqref{fraceq1} with fractional order $0 < \alpha < 1$. By application of the classical Picard-Lindel\"of successive approximation scheme and the standard Gronwall-Bellman inequality, existence and uniqueness of the solution to Eq.~\eqref{fraceq1} with fractional order $\frac{1}{2} < \alpha < 1$ is discussed in~\cite{pedjeu2012stochastic}. However, the case with $0 < \alpha \leq \frac{1}{2}$ remains to be investigated. We can apply the generalized Gronwall-Bellman inequality developed in Section~\ref{mainineq} to derive existence and uniqueness of the solution to Eq.~\eqref{fraceq1} when $0 < \alpha < 1$. 

\begin{thm}\label{existenceanduniqueness}
	Let $0 < \alpha < 1, T > 0$, and $B_{t}$ be a $m-$dimensional Brownian motion on a complete probability space $\Omega\equiv (\Omega, \mathcal{F}, \mathbb{P})$. Assume that $b(\cdot, \cdot), \sigma_{1}(\cdot, \cdot): [0, T]\times\mathbb{R}^{n}\to\mathbb{R}^{n}, \sigma_{2}(\cdot, \cdot): [0, T]\times\mathbb{R}^{n}\to\mathbb{R}^{n\times m}$ are measurable functions satisfying the linear growth condition,
	\begin{eqnarray}\label{linearcond}
		|b(t, x)|^{2} + |\sigma_{1}(t, x)|^{2} + |\sigma_{2}(t, x)|^{2} \leq K^{2}(1 + |x|^{2}),
	\end{eqnarray}
	for some constant $K > 0$ and the Lipschitz condition,
	\begin{eqnarray}\label{lipschitzcond}
		|b(t, x) - b(t, y)| + |\sigma_{1}(t, x) - \sigma_{1}(t, y)| + |\sigma_{2}(t, x) - \sigma_{2}(t, y)| \leq L|x - y|,
	\end{eqnarray}
	for some constant $L > 0$. Let $x_{0}$ be a random variable, which is independent of the $\sigma-$algebra $\mathcal{F}_{t}\subset\mathcal{F}_{\infty}$ generated by $\{B_{t}, t\geq 0\}$ and satisfies $\mathbb{E}|x_{0}|^{2} < \infty$. Then, the fractional stochastic  differential equation Eq.~\eqref{fraceq1} has a unique $t-$continuous solution~$x(t, \omega)$ with the property that $x(t, \omega)$ is adapted to the filtration~$\mathcal{F}_{t}^{x_{0}}$ generated by $x_{0}$ and $\{B_{t}, t\geq 0\}$, and 
	\begin{eqnarray*}
		\mathbb{E}\Bigg[\int_{0}^{T}|x(t)|^{2}\mathrm{d}t\Bigg] < \infty.
	\end{eqnarray*}
\end{thm}

\begin{proof}
	\textbf{(Existence)} From Eq.~\eqref{fracinteg}, the corresponding equivalent stochastic integral equation of the fractional stochastic  differential equation Eq.~\eqref{fraceq1} is rewritten as
	\begin{eqnarray*}
	x(t) = x_{0} + \int_{0}^{t}b(s, x(s))\mathrm{d}s + \alpha\int_{0}^{t}(t - s)^{\alpha - 1}\sigma_{1}(s, x(s))\mathrm{d}s + \int_{0}^{t}\sigma_{2}(s, x(s))\mathrm{d}B_{s},
	\end{eqnarray*}
	where $0 \leq t < T$ and $0 < \alpha < 1$. For more details about this equivalence between Eq.~\eqref{fraceq1} and Eq.~\eqref{fracinteg}, we refer to~\cite{jumarie2005solution, jumarie2006new, jumarie2005representation}. By the method of Picard-Lindel\"of successive approximations, define $x^{0}(t) = x_{0}$ and $x^{k}(t) = x^{k}(t, \omega)$ inductively as follows
	\begin{eqnarray}\label{picarditeration}
		x^{k+1}(t) = x_{0} + \int_{0}^{t}b(s, x^{k}(s))\mathrm{d}s + \alpha\int_{0}^{t}(t - s)^{\alpha - 1}\sigma_{1}(s, x^{k}(s))\mathrm{d}s + \int_{0}^{t}\sigma_{2}(s, x^{k}(s))\mathrm{d}B_{s}.
	\end{eqnarray}
	Applying the inequality $|x + y + z|^{2} \leq 3|x|^{2} + 3|y|^{2} + 3|z|^{2}$ leads to
		\begin{eqnarray*}\label{triangularineq}
		\begin{aligned}
		\mathbb{E}|x^{k+1}(t) - x^{k}(t)|^{2} &\leq 3\mathbb{E}\bigg|\int_{0}^{t}\bigg(b(s, x^{k}(s)) - b(s, x^{k-1}(s))\bigg)\mathrm{d}s\bigg|^{2}\\ 
		&+ 3\mathbb{E}\bigg|\alpha\int_{0}^{t}(t - s)^{\alpha - 1}\bigg(\sigma_{1}(s, x^{k}(s)) - \sigma_{1}(s, x^{k-1}(s))\bigg)\mathrm{d}s\bigg|^{2}\\ 
		&+ 3\mathbb{E}\bigg|\int_{0}^{t}\bigg(\sigma_{2}(s, x^{k}(s)) - \sigma_{2}(s, x^{k-1}(s))\bigg)\mathrm{d}s\bigg|^{2}\\
		&:= I_{1} + I_{2} + I_{3}.
		\end{aligned}
		\end{eqnarray*}
	Using the Cauchy-Schwartz inequality on the first two terms, $I_{1}$ and $I_{2}$, plus It\^o's Isometry, see in~\cite{oksendal2013stochastic}, in the third term, $I_{3}$, produces
		\begin{eqnarray*}\label{triangularineq}
			\begin{aligned}
				\mathbb{E}|x^{k+1}(t) - x^{k}(t)|^{2} &\leq 3T\mathbb{E}\int_{0}^{t}\bigg(b(s, x^{k}(s)) - b(s, x^{k-1}(s))\bigg)^{2}\mathrm{d}s\\ 
				&+ 3\alpha^{2}\int_{0}^{t}(t - s)^{\alpha - 1}\mathrm{d}s\mathbb{E}\int_{0}^{t}(t - s)^{\alpha - 1}\bigg(\sigma_{1}(s, x^{k}(s)) - \sigma_{1}(s, x^{k-1}(s))\bigg)^{2}\mathrm{d}s\\ 
				&+ 3
				\mathbb{E}\int_{0}^{t}\bigg(\sigma_{2}(s, x^{k}(s)) - \sigma_{2}(s, x^{k-1}(s))\bigg)^{2}\mathrm{d}s\\
				&:= J_{1} + J_{2} + J_{3}.
			\end{aligned}
		\end{eqnarray*}
Finally, using the Lipschitz condition Eq.~\eqref{lipschitzcond} on all terms, $J_{1}, J_{2}, J_{3}$, evaluating the first integral in the second term, $J_{2}$, and combining the first and third terms, $J_{1}$ and $J_{3}$, yields
	\begin{eqnarray}\label{iterineqofappl}
		\begin{aligned}
			\mathbb{E}|x^{k+1}(t) - x^{k}(t)|^{2} &\leq 3L^{2}(1+T)\int_{0}^{t}\mathbb{E}|x^{k}(s) - x^{k-1}(s)|^{2}\mathrm{d}s\\ 
			&+ 3L^{2}(1+T)\int_{0}^{t}(t - s)^{\alpha - 1}\mathbb{E}|x^{k}(s) - x^{k-1}(s)|^{2}\mathrm{d}s.
		\end{aligned}
	\end{eqnarray}
	Thus, for locally integrable function~$\phi(t)$, define an operator~$B$ as follows
	\begin{eqnarray*}
		B\phi(t) := 3L^{2}(1+T)\Bigg\{\int_{0}^{t}\phi(s)\mathrm{d}s + \int_{0}^{t}(t - s)^{\alpha - 1}\phi(s)\mathrm{d}s\Bigg\}.
	\end{eqnarray*}
	Then, iterating Eq.~\eqref{iterineqofappl} yields
	\begin{eqnarray*}
		\mathbb{E}|x^{k+1}(t) - x^{k}(t)|^{2} \leq B(\mathbb{E}|x^{k}(t) - x^{k-1}(t)|^{2}) \leq \cdots \leq B^{k}(\mathbb{E}|x^{1}(t) - x^{0}(t)|^{2}).
	\end{eqnarray*}
	Since $0 < \alpha < 1$ and $\mathbb{E}|x^{1}(t) - x^{0}(t)|^{2}$ is nonnegative and locally integrable, from the first claim, Eq.~\eqref{guessiter}, and the Eq.~\eqref{boundofoperator} in the proof of the second claim in Section~\ref{mainineq}, we know that 
	\begin{eqnarray*}
		\begin{aligned}
			\mathbb{E}|x^{k+1}(t) - x^{k}(t)|^{2} &\leq B^{k}(\mathbb{E}|x^{1}(t) - x^{0}(t)|^{2})\\ &\leq  \frac{[\Gamma(\alpha)]^{k}\max\{t^{k\alpha - 1}, t^{k}\}}{\Gamma(k\alpha)}[6L^{2}(1+T)]^{k}\int_{0}^{t}\mathbb{E}|x^{1}(s) - x^{0}(s)|^{2}\mathrm{d}s. 
		\end{aligned}
	\end{eqnarray*}	
Similarly, apply the Cauchy-Schwartz inequality, the It\^o's Isometry, and the linear growth condition, Eq.~\eqref{linearcond}, instead of Lipschitz condition, Eq.~\eqref{lipschitzcond}, to compute
	\begin{eqnarray*}
		\mathbb{E}|x^{1}(t) - x^{0}(t)|^{2} \leq 3(1+T)K^{2}(1 + \mathbb{E}|x_{0}|^{2})(t + t^{\alpha}).
	\end{eqnarray*}
This implies 
	\begin{eqnarray}\label{supoutside}
		\sup_{0 \leq t \leq T}\mathbb{E}|x^{k+1}(t) - x^{k}(t)|^{2} \leq M_{0} \frac{[\Gamma(\alpha)]^{k}\max\{T^{k\alpha - 1}, T^{k}\}}{\Gamma(k\alpha)}[6L^{2}(1+T)]^{k},
	\end{eqnarray}
	where $M_{0} = 3(1+T)K^{2}(1 + \mathbb{E}|x_{0}|^{2})\bigg(\frac{T^{2}}{2} + \frac{T^{\alpha + 1}}{\alpha + 1}\bigg)$ is independent of $k$ and $t$. Thus, for any $m > n > 0$, 
\begin{eqnarray*}
	\begin{aligned}
		\|x^{m}(t) - x^{n}(t)\|^{2}_{L^{2}(\mathbb{P})} &\leq \sum_{k=n}^{m}\|x^{k+1}(t) - x^{k}(t)\|^{2}_{L^{2}(\mathbb{P})} = \sum_{k=n}^{m}\int_{0}^{T}\mathbb{E}|x^{k+1}(t) - x^{k}(t)|^{2}\mathrm{d}t\\
		& \leq M_{1}\sum_{k=n}^{m}\frac{[\Gamma(\alpha)]^{k}\max\{T^{k\alpha - 1}, T^{k}\}}{\Gamma(k\alpha)}[6L^{2}(1+T)]^{k}\\
		&  \to 0, ~as~m,~n\to\infty,
	\end{aligned}
\end{eqnarray*}
where $M_{1} = 3(1+T)K^{2}(1 + \mathbb{E}|x_{0}|^{2})\bigg(\frac{T^{3}}{2} + \frac{T^{\alpha + 2}}{\alpha + 1}\bigg)$ is independent of $k$ and $t$.
This means the successive approximations~$(x^{k}(t))$ are mean-square convergent uniformly on $[0, T]$.
It remains now to show that the sequence of successive approximations~$(x^{k}(t))$ is almost surely convergent. First, apply Chebyshev's inequality to yield
\begin{eqnarray*}
	\begin{aligned}
	    \sum_{k = 1}^{\infty}\mathbb{P}\bigg\{\sup_{0 \leq t \leq T}|x^{k+1}(t) - x^{k}(t)| > \frac{1}{k^{2}}\bigg\} &\leq \sum_{k=1}^{\infty}k^{4}\mathbb{E}\bigg(\sup_{0\leq t\leq T}|x^{k+1}(t) - x^{k}(t)|\bigg)^{2}\\
	    &= \sum_{k=1}^{\infty}k^{4}\mathbb{E}\bigg(\sup_{0\leq t\leq T}|x^{k+1}(t) - x^{k}(t)|^{2}\bigg).		
	\end{aligned}
\end{eqnarray*}
By computations similar to those leading to Eq.~\eqref{supoutside} and Doob's Maximal Inequality for martingales, 
\begin{eqnarray*}
	\sum_{k = 1}^{\infty}\mathbb{P}\bigg\{\sup_{0 \leq t \leq T}|x^{k+1}(t) - x^{k}(t)| > \frac{1}{k^{2}}\bigg\} \leq M_{0}\sum_{k=1}^{\infty}\frac{[\Gamma(\alpha)]^{k}\max\{T^{k\alpha - 1}, T^{k}\}}{\Gamma(k\alpha)}[6L^{2}(1+T)]^{k}k^{4},
\end{eqnarray*}
which is finite. Then, applying the Borel-Cantelli lemma yields, 
\begin{eqnarray*}
	\mathbb{P}\bigg\{\sup_{0 \leq t \leq T}|x^{k+1}(t) - x^{k}(t)| > \frac{1}{k^{2}} ~\textrm{for infinitely many}~k\bigg\} = 0.
\end{eqnarray*}
So there exists a random variable~$x(t)$, which is the limit of the following sequence
\begin{eqnarray*}
	x^{k}(t) = x^{0}(t) + \sum_{n = 0}^{k-1}(x^{n+1}(t) - x^{n}(t))\to x(t) ~a.s.,
\end{eqnarray*}
uniformly on $[0, T]$. Also $x(t)$ is $t-$continuous since $x^{k}(t)$ is $t-$continuous for all $k$. Therefore, taking the limit on both sides of Eq.~\eqref{picarditeration} as $k\to\infty$, there is a stochastic process $x(t)$ satisfying Eq.~\eqref{fracinteg}.

\textbf{(Uniqueness)} The uniqueness is due to the It\^o Isometry and the Lipschitz condition, Eq.~\eqref{lipschitzcond}. Let $x_{1}(t) = x_{1}(t,\omega)$ and $x_{2}(t) = x_{2}(t, \omega)$ be solutions of Eq.~\eqref{fracinteg}, which have the initial values,~$x_{1}(0) = y_{1}$ and $x_{2}(0) = y_{2}$, respectively. Similarly, apply the Cauchy-Schwartz inequality, the It\^o Isometry, and the Lipschitz condition Eq.~\eqref{lipschitzcond} to compute
\begin{eqnarray*}
	\begin{aligned}
		\mathbb{E}|x_{1}(t) - x_{2}(t)|^{2} &\leq 4\mathbb{E}|y_{1} - y_{2}|^{2} + 4L^{2}(1 + T)\int_{0}^{t}\mathbb{E}|x_{1}(s) - x_{2}(s)|^{2}\mathrm{d}s\\ 
		&+ 4\alpha L^{2}T^{\alpha}\int_{0}^{t}(t - s)^{\alpha - 1}\mathbb{E}|x_{1}(s) - x_{2}(s)|^{2}\mathrm{d}s.	\end{aligned}
\end{eqnarray*}
By application of the generalized Gronwall-Bellman inequality in Corollary~\ref{consisgronwall}, we have
\begin{eqnarray*}
	\mathbb{E}|x_{1}(t) - x_{2}(t)|^{2} \leq 4\mathbb{E}|y_{1} - y_{2}|^{2}E_{\alpha}(4\alpha L^{2}T^{\alpha}\Gamma(\alpha)t^{\alpha})\exp(\frac{1}{\alpha}4L^{2}(1 + T)t).
\end{eqnarray*}
Since $x_{1}(t)$ and $x_{2}(t)$ both satisfy the stochastic integral equation Eq.~\eqref{fracinteg}, the initial values~$y_{1}$ and $y_{2}$ are both equal to $x_{0}$. This means $\mathbb{E}|x_{1}(t) - x_{2}(t)|^{2} = 0$ for all $t > 0$. Furthermore, 
\begin{eqnarray*}
	\mathbb{P}\bigg\{|x_{1}(t) - x_{2}(t)| = 0,~\textrm{for all}~0 \leq t \leq T\bigg\} = 1.
\end{eqnarray*}
Therefore, the uniqueness of the solution to Eq.~\eqref{fracinteg} is proved.
\end{proof}
\section{Fractional Fokker-Planck-Kolmogorov Equation}\label{fpe}
Based on the existence and uniqueness Theorem~\ref{existenceanduniqueness} developed in Section~\ref{application}, we derive the fractional Fokker-Planck-Kolmogorov equation associated to the unique solution of the fractional SDE, Eq.~\eqref{fraceq1}. Before deriving the fractional Fokker-Planck-Kolmogorov equation, we first introduce an It\^o formula from \cite{pedjeu2012stochastic} to the following It\^o process 
\begin{eqnarray}\label{fraceq14}
x(t) = x_{0} + \int_{0}^{t}b(s,x(s))\mathrm{d}s + \int_{0}^{t}\sigma_{1}(s, x(s))\mathrm{d}s^{\alpha} + \int_{0}^{t}\sigma_{2}(s, x(s))\mathrm{d}B_{s},
\end{eqnarray}
where $0 < \alpha < 1$, $B_{t}$ is the $m-$dimensional standard Brownian motion, and functions $b, \sigma_{1}, \sigma_{2}$ satisfy the conditions in Theorem~\ref{existenceanduniqueness}.
\begin{lem}\label{itoformula}
	Let $X(t)$ satisfy the Eq.~\eqref{fraceq14} and furthermore, let $V\in C[R^{+}\times R^{n}, R^{m}]$, and assume that $V_{t}$, $V_{x}$, $V_{xx}$ exist and continuous for $(t, x)\in R^{+}\times R^{n}$, where $V_{x}(t, x)$ is an $m\times n$ Jacobian matrix of $V(t, x)$ and $V_{xx}(t, x)$ is an $m\times n$ Hessian matrix. Then,
	\begin{eqnarray*}
		\mathrm{d}V(t, X(t)) = L_{1}V(t, X(t))\mathrm{d}t + L_{2}V(t, X(t))\mathrm{d}t^{\alpha} + L_{3}V(t, X(t))\mathrm{d}B_{t},
	\end{eqnarray*}
	where 
\begin{eqnarray*}
	L_{1}V(t, x) = V_{t}(t, x) + V_{x}(t, x)b(t, x) + \frac{1}{2}\sigma_{2}(t, x)^{T}V_{xx}(t, x)\sigma_{2}(t, x)
\end{eqnarray*}
and
\begin{eqnarray*}
	L_{2}V(t, x) = V_{x}(t, x)\sigma_{1}(t, x), ~~~L_{3}V(t, x) = V_{x}(t, x)\sigma_{2}(t, x).
\end{eqnarray*}
\end{lem}

By applying the existence and uniqueness Theorem~\ref{existenceanduniqueness} and It\^o's formula, Lemma~\ref{itoformula}, the following fractional Fokker-Planck-Kolmogorov equation is established.
\begin{thm}
	Let $B(t)$ be the $m-$dimensional standard Brownian motion. Suppose that $X(t)$ is the solution to the fractional SDE Eq.~\eqref{fraceq1} whose coefficient functions $b, \sigma_{1}$ and $\sigma_{2}$ satisfy the conditions in Theorem~\ref{existenceanduniqueness}. Then the transition probabilities $P^{X}(t, x) = P^{X}(t, x|0, x_{0})$ of $X(t)$ satisfy the following fractional type differential equation 
	\begin{eqnarray}\label{fracfokkplanck}
		\mathrm{d}P^{X}(t, x) = A_{x}^{*}P^{X}(t, x)\mathrm{d}t + B_{x}^{*}P^{X}(t, x)\mathrm{d}t^{\alpha}
	\end{eqnarray}
	with initial condition $P^{X}(0, x) = \delta_{x_{0}}(x)$, the Dirac delta function with mass on $x_{0}$, and $A^{*}_{x}, B^{*}_{x}$ are spatial operators defined respectively by
	\begin{eqnarray*}
		A^{*}_{x}h(x) = -\sum_{i=1}^{n}\frac{\partial}{\partial x_{i}}[b_{i}(t, x)h(x)] +  \frac{1}{2}\sum_{i=1}^{j=1}\sum_{i=1}^{j=1}\frac{\partial^{2}}{\partial x_{i}\partial x_{j}}[\sum_{k=1}^{m}\delta_{2}^{ik}\delta_{2}^{jk}(t, x)h(x)]
	\end{eqnarray*}
and 
\begin{eqnarray*}
	B_{x}^{*} = -\sum_{i=1}^{n}\frac{\partial}{\partial x_{i}}[\delta_{1}^{i}h(x)],
\end{eqnarray*}
where $b = (b_{1}, \cdots, b_{n})^{T}, \delta_{1} = (\delta_{1}^{1}, \cdots, \delta_{1}^{n})^{T}$ and $\delta_{2}$ is an $n\times m$ matrix with elements $[\delta_{2}]_{ij} = \delta_{2}^{ij}$.
\end{thm}
\begin{proof}
Let $f\in C_{c}^{\infty}(R^{n})$, i.e. $f$ is an infinitely differential function on $R^{n}$ with compact support. Since $X(t)$ is the solution of the stochastic fractional differential equation Eq.~\eqref{fraceq1}, this means $X(t)$ satisfies the stochastic integral equation Eq.~\eqref{fraceq14}. So apply It\^o formula Lemma~\ref{itoformula} on $f(X(t))$ to yield
\begin{eqnarray}\label{itoequation}
\begin{aligned}
f(X(t)) - f(x_{0}) &= \int_{0}^{t}\bigg(f_{x}(X(s))b(s, X(s)) + \frac{1}{2}\sigma_{2}^{T}(s, X(s))f_{xx}(X(s))\sigma_{2}(s, X(s))\bigg)\mathrm{d}s\\
&+ \int_{0}^{t}f_{x}(X(s))\sigma_{1}(s, X(s))\mathrm{d}s^{\alpha} + \int_{0}^{t}f_{x}(X(s))\sigma_{2}(s, X(s))\mathrm{d}B_{s}.
\end{aligned}
\end{eqnarray}
Notice the fact that
\begin{eqnarray*}
	\int_{0}^{t}f_{x}(X(s))\sigma_{1}(s, X(s))\mathrm{d}s^{\alpha} = \alpha\int_{0}^{t}(t - s)^{\alpha - 1}f_{x}(X(s))\sigma_{1}(s, X(s))\mathrm{d}s,
\end{eqnarray*}
and more details on this equality can be found in \cite{jumarie2005solution, pedjeu2012stochastic}. Thus Eq.~\eqref{itoequation} can be written as
\begin{eqnarray}\label{itoequation1}
\begin{aligned}
f(X(t)) - f(x_{0}) &= \int_{0}^{t}\bigg(f_{x}(X(s))b(s, X(s)) + \frac{1}{2}\sigma_{2}^{T}(s, X(s))f_{xx}(X(s))\sigma_{2}(s, X(s))\bigg)\mathrm{d}s\\
&+ \alpha\int_{0}^{t}(t - s)^{\alpha - 1}f_{x}(X(s))\sigma_{1}(s, X(s))\mathrm{d}s + \int_{0}^{t}f_{x}(X(s))\sigma_{2}(s, X(s))\mathrm{d}B_{s}.
\end{aligned}
\end{eqnarray}
Since the integral $\int_{0}^{t}f_{x}(X(s))\sigma_{2}(s, X(s))\mathrm{d}B_{s}$ is a martingale with respect to the filtration $\mathcal{F}_{t}$, take conditional expectations on both sides of Eq.~\eqref{itoequation1} to obtain
\begin{eqnarray}\label{conditionalexpectation}
\begin{aligned}
\mathbb{E}[f(X(t))|X(0) = x_{0}] - f(x_{0}) &= \mathbb{E}\bigg[\int_{0}^{t}f_{x}(X(s))b(s, X(s))\mathrm{d}s\bigg|X(0) = x_{0}\bigg]\\ &+ \mathbb{E}\bigg[\frac{1}{2}\int_{0}^{t}\sigma_{2}^{T}(s, X(s))f_{xx}(X(s))\sigma_{2}(s, X(s))\mathrm{d}s\bigg|X(0) = x_{0}\bigg]\\
&+ \mathbb{E}\bigg[\alpha\int_{0}^{t}(t - s)^{\alpha - 1}f_{x}(X(s))\sigma_{1}(s, X(s))\mathrm{d}s\bigg|X(0) = x_{0}\bigg].
\end{aligned}	
\end{eqnarray}
By Fubini's Theorem and integration by parts, the above Eq.~\eqref{conditionalexpectation} can be rewritten as
\begin{eqnarray*}
	\begin{aligned}
		\int_{R^{n}}f(x)P^{X}(t, x)\mathrm{d}x - f(x_{0}) 
		&= \int_{0}^{t}\int_{R^{n}}f_{x}(x)b(s, x)P^{X}(s, x)\mathrm{d}x\mathrm{d}s\\
		&+ \frac{1}{2}\int_{0}^{t}\int_{R^{n}}\sigma_{2}^{T}(s, x)f_{xx}(x)\sigma_{2}(s, x)P^{X}(s, x)\mathrm{d}x\mathrm{d}s\\ 
		&+ \alpha\int_{0}^{t}\int_{R^{n}}(t - s)^{\alpha - 1}f_{x}(x)\sigma_{1}(s, x)P^{X}(s, x)\mathrm{d}x\mathrm{d}s\\
		&= \int_{R^{n}}f(x)\int_{0}^{t}A^{*}_{1}P^{X}(s, x)\mathrm{d}s\mathrm{d}x\\ 
		&+ \int_{R^{n}}f(x)\alpha\int_{0}^{t}(t - s)^{\alpha - 1}A^{*}_{2}P^{X}(s, x)\mathrm{d}s\mathrm{d}x.
	\end{aligned}
\end{eqnarray*}
Since $f\in C_{c}^{\infty}(R^{n})$ is arbitrary and $ C_{c}^{\infty}(R^{n})$ is dense in $L^{2}(R^{n})$, 
\begin{eqnarray}\label{integralequationoffokkerplanckeq}
	P^{X}(t, x) - \delta_{x_{0}}(x) = \int_{0}^{t}A^{*}_{1}P^{X}(s, x)\mathrm{d}s + \alpha\int_{0}^{t}(t - s)^{\alpha - 1}A^{*}_{2}P^{X}(s, x)\mathrm{d}s,
\end{eqnarray}
where $\delta_{x_{0}}(x)$ is a generalized function taking value $\delta_{x_{0}}(x) = P^{X}(0,x_{0})$. Finally, take the derivative with respect to time $t$ on both sides of Eq.~\eqref{integralequationoffokkerplanckeq} to yield the desired result Eq.~\eqref{fracfokkplanck}.
\end{proof}
\section{Conclusion}
In this paper, a new type of Gronwall-Bellman inequality is established for a class of integral equations with a mixture of nonsingular and singular integrals. This new type of Gronwall-Bellman inequality can be considered as a generalization of known Gronwall-Bellman inequalities dealing with an integral equation having nonsingular or singular integrals, separately. With this new type of Gronwall-Bellman inequality, existence and uniqueness of the solution to a fractional SDE with fractional oder $0 < \alpha < 1$ is investigated. Furthermore, based on the existence and uniqueness result, a fractional type Fokker-Planck-Kolmogorov equation associated to the solution of a fractional SDE is derived.
\section{Acknowledgments}
The author wishes to thank Dr. Marjorie Hahn for her advice, fruitful discussion, encouragement and patience with my research, and Dr. Xiaozhe Hu for his helpful discussion. 

\section*{References}
\bibliography{mybibfile}

\end{document}